\documentclass{article}
\usepackage{amsfonts,amssymb,amsmath,amsthm}
\usepackage{url}
\usepackage{geometry}
\usepackage{graphicx}	

\def\thtext#1{
  \catcode`@=11
  \gdef\@thmcountersep{. #1}
  \catcode`@=12
}

\def\threst{
  \catcode`@=11
  \gdef\@thmcountersep{.}
  \catcode`@=12
}

\def\ig#1#2#3#4{\begin{figure}[!ht]\begin{center}%
\includegraphics[height=#2\textheight]{#1.eps}\caption{#4}\label{#3}%
\end{center}\end{figure}}

\theoremstyle{plain}
\newtheorem{thm}{Theorem}
\newtheorem{prop}{Proposition}[section]
\newtheorem{cor}[prop]{Corollary}
\newtheorem{lem}[prop]{Lemma}

\theoremstyle{definition}
\newtheorem{examp}[prop]{Example}
\newtheorem{dfn}[prop]{Definition}
\newtheorem{rk}[prop]{Remark}
\newtheorem{constr}[prop]{Construction}

 \catcode`@=11
 \def\.{.\spacefactor\@m}
 \catcode`@=12

\newcommand{\e}{\varepsilon}

\newcommand{\G}{\Gamma}

\renewcommand{\r}{\rho}
\newcommand{\s}{\sigma}

\newcommand{\cC}{\mathcal{C}}

\newcommand{\cH}{\mathcal{H}}
\newcommand{\cM}{\mathcal{M}}
\newcommand{\cP}{\mathcal{P}}
\newcommand{\cR}{\mathcal{R}}

\newcommand{\N}{\mathbb{N}}

\newcommand{\R}{\mathbb{R}}

\newcommand{\hR}{{\hat R}}
\newcommand{\hx}{{\hat x}}
\newcommand{\hX}{{\hat X}}
\newcommand{\hy}{{\hat y}}
\newcommand{\hY}{{\hat Y}}

\renewcommand{\c}{\circ}
\renewcommand{\:}{\colon}
\newcommand{\0}{\emptyset}

\newcommand{\sm}{\setminus}
\newcommand{\oPi}{\stackrel{\raise-2pt\hbox{$\c$}}\Pi}
\newcommand{\oW}{\stackrel{\raise-2pt\hbox{$\c$}}W}
\renewcommand{\sp}{\supset}
\renewcommand{\ss}{\subset}
\newcommand{\x}{\times}

\newcommand{\diam}{\operatorname{diam}}
\newcommand{\dis}{\operatorname{dis}}
\newcommand{\im}{\operatorname{im}}
\newcommand{\Int}{\operatorname{Int}}
\newcommand{\mf}{\operatorname{mf}}
\newcommand{\opt}{{\operatorname{opt}}}
\newcommand{\smt}{\operatorname{smt}}
\newcommand{\SMT}{\operatorname{SMT}}
\newcommand{\ssr}{\operatorname{ssr}}

\title{Gromov--Hausdorff Distance, Irreducible Correspondences, Steiner Problem, and Minimal Fillings}
\author{Alexander O. Ivanov, Alexey A. Tuzhilin}
\date{}

\begin{document}
\maketitle
\begin{abstract}
We introduce irreducible correspondences that enables us to calculate the Gromov--Hausdorff distances effectively. By means of these correspondences, we show that the set of all metric spaces each consisting of no more than $3$ points is isometric to a polyhedral cone in the space $\R^3$ endowed with the maximum norm. We prove that for any $3$-point metric space such that all the triangle inequalities are strict in it, there exists a neighborhood such that the Steiner minimal trees (in Gromov--Hausdorff space) with boundaries from this neighborhood are minimal fillings, i.e., it is impossible to decrease the lengths of these trees by isometrically embedding their boundaries into any other ambient metric space. On the other hand, we construct an example of $3$-point boundary whose points are $3$-point metric spaces such that its Steiner minimal tree in the Gromov--Hausdorff space is not a minimal filling. The latter proves that the Steiner subratio of the Gromov--Hausdorff space is less than 1. The irreducible correspondences enabled us to create a quick algorithm for calculating the Gromov--Hausdorff distance between finite metric spaces. We carried out a numerical experiment and obtained more precise upper estimate on the Steiner subratio: we have shown that it is less than $0.857$.
\end{abstract}

\section*{Introduction}
\markright{\thesection.~Introduction}
The Gromov--Hausdorff distance between any metric spaces measures their ``degree of similarity'' (see below for exact definition). Explicit calculation of this distance is a nontrivial problem. There exists a few methods making possible to progress in this direction. One of them is based on the technique of correspondences.

In the present paper we discuss this technique, describe some examples of its application, define so-called irreducible correspondences, and show how the latter ones can be used for effective calculations of the Gromov--Hausdorff distance.

As a consequence we show that the family of $3$-point metric spaces is isometric to a polyhedral cone in $\R^3_\infty$, where $\R^3_\infty$ stands for the space $\R^3$ with the norm $\bigl\|(x^1,x^2,x^3)\bigr\|=\max|x^i|$. This isometry leads to some other results related with the Steiner problem.

Recall that a \emph{Steiner minimal tree}, or a \emph{shortest tree}, on a finite subset $M$ of a metric space $X$, is a tree of the least possible length among all the trees constructed on finite sets $V\ss X$, $V\sp M$, see~\cite{IvaTuz} or~\cite{HwRiW} for further discussion. Such set $M$ is called the \emph{boundary set\/} of this shortest tree. Notice that such a tree may not exist for some $M$, although ``the length'' of such a tree (more exactly, the infimum of the lengths of the trees on various $V$) is always defined. Since $V$ may contain additional points, the lengths of Steiner minimal trees depend not only on the distances between the points from $M$, but also on the geometry of the ambient space $X$. If we embed isometrically the space $M$ into various metric spaces $X$ and minimize the lengths of the corresponding shortest trees over all such embeddings, we get the value which is called the \emph{length of minimal filling for $M$}, see~\cite{ITMinFil}.

One of rather interesting problems is to describe the metric spaces for which the length of each shortest tree coincides with the length of minimal filling for the boundary of this tree. An example of such spaces is $\R^n_\infty$ that is proved by Z.\,N.~Ovsyannikov~\cite{Ovs}. Complete classification of Banach spaces satisfying this property together with the additional restriction that each finite subset is connected by some shortest tree is obtained by B.\,B.~Bednov and P.\,A.~Borodin~\cite{BB}.

The isometry between the family of $3$-point metric spaces and the cone in $\R^3_\infty$ has led us to conjecture that the lengths of the Steiner minimal trees in the space $\cM$ of isometry classes of compact metric spaces with Gromov--Hausdorff distance may be equal to the lengths of the corresponding minimal fillings.

To our surprise, we constructed a counterexample to this conjecture: it turns out that even in the case of $3$-point metric spaces, a Steiner minimal tree may be longer than the minimal filling for its boundary. In~\cite{ITMinFil} we introduced the concept of Steiner subratio which measures in some sense a ``variational curvature'' of the ambient space, i.e., how many times the minimal fillings for a finite subset of the space can be shorter than the Steiner minimal tree with the same boundary. Thus, in the present paper we show that the Steiner subratio of the Gromov--Hausdorff space is less than $1$; more precisely, it is less than $0.857$.

\section{Preliminaries}
\markright{\thesection.~Preliminaries}
Let $X$ be an arbitrary metric space. By $|xy|$ we denote the distance between its points $x$ and $y$. Let $\cP(X)$ be the family of all nonempty subsets of the space $X$. For $A,\,B\in\cP(X)$ put
$$
d_H(A,B)=\max\bigl\{\sup_{a\in A}\inf_{b\in B}|ab|,\,\sup_{b\in B}\inf_{a\in A}|ab|\bigr\}.
$$
The value $d_H(A,B)$ is called the \emph{Hausdorff distance between $A$ and $B$}.

Let $\cH(X)\ss\cP(X)$ be the family of all nonempty closed bounded subsets of $X$.

\begin{prop}[\cite{BurBurIva}]
The restriction of $d_H$ onto $\cH(X)$ is a metric.
\end{prop}

Let $X$ and $Y$ be metric spaces. A triple $(X',Y',Z)$ consisting of a metric space $Z$ and its subsets $X'$ and $Y'$ which are isometric to $X$ and $Y$, respectively, is called a \emph{realization of the pair $(X,Y)$}. \emph{The Gromov--Hausdorff distance $d_{GH}(X,Y)$ between $X$ and $Y$} is the infimum of the numbers $r$ for which there exists a realization $(X',Y',Z)$ of the pair $(X,Y)$ such that $d_H(X',Y')\le r$.

By $\cM$ we denote  the family of all compact metric spaces (considered up to isometry) endowed with the Gromov--Hausdorff metric. This $\cM$ is often called the \emph{Gromov--Hausdorff space}.

\begin{prop}[\cite{BurBurIva}]
The restriction of $d_{GH}$ onto $\cM$ is a metric.
\end{prop}

Now we give an equivalent definition of the Gromov--Hausdorff distance which is more suitable for specific calculations. Recall that a \emph{relation\/} between the sets $X$ and $Y$ is a subset of the Cartesian product $X\x Y$.  By $\cP(X,Y)$ we denote the set of all nonempty relations between $X$ and $Y$. If $\pi_X\:X\x Y\to X$ and $\pi_Y\:X\x Y\to Y$ are canonical projections, i.e., $\pi_X(x,y)=x$ and $\pi_Y(x,y)=y$, then we denote their restrictions onto each relation $\s\in\cP(X,Y)$ in the same way.

A relation $R$ between $X$ and $Y$ is called a \emph{correspondence}, if the restrictions of the canonical projections $\pi_X$ and $\pi_Y$ onto $R$ are surjections. By $\cR(X,Y)$ we denote the set of all correspondences between $X$ and $Y$. 

Let $X$ and $Y$ be metric spaces, then for each relation $\s\in\cP(X,Y)$ its \emph{distortion $\dis\s$} is defined as follows:
$$
\dis\s=\sup\Bigl\{\bigl||xx'|-|yy'|\bigr|: (x,y)\in\s,\ (x',y')\in\s\Bigr\}.
$$

The following result is well-known.

\begin{prop}[\cite{BurBurIva}]\label{th:GH-metri-and-relations}
For any metric spaces $X$ and $Y$ it holds
$$
d_{GH}(X,Y)=\frac12\inf\bigl\{\dis R\mid R\in\cR(X,Y)\bigr\}.
$$
\end{prop}

\begin{dfn}
A correspondence $R\in\cR(X,Y)$ is called \emph{optimal\/} if $d_{GH}(X,Y)=\frac12\dis R$. By $\cR_\opt(X,Y)$ we denote the set of all optimal correspondences between $X$ and $Y$.
\end{dfn}

\begin{prop}[\cite{IvaIliadisTuz}]\label{prop:optimal-correspondence-exists}
For any $X,\,Y\in\cM$ we have $\cR_\opt(X,Y)\ne\0$.
\end{prop}

\section{Correspondences and their properties}
\markright{\thesection.~Correspondences and their properties}
In what follows, unless otherwise is stated, $X$ and $Y$ always stand for some nonempty sets. Let us consider each relation $\s\in\cP(X,Y)$ as a multivalued mapping whose domain may be smaller than the entire $X$. Then, similarly to the customary case of mappings, for each $x\in X$  its \emph{image\/} $\s(x)=\{y\in Y:(x,y)\in\s\}$ is defined, and for each $A\ss X$ put  $\s(A)$ to be equal to the union of the images of all elements from $A$. Further, for each $y\in Y$ its \emph{preimage} $\s^{-1}(y)=\{x\in X:(x,y)\in\s\}$ is defined, and for each $B\ss Y$  define its preimage as the union of the preimages of all its elements.

Let $\s\in\cP(X,Y)$ be an arbitrary relation. For each $x\in X$ we define its \emph{multiplicity $k_\s(x)$ in $\s$} as the cardinality of the set $\s(x)$. Similarly, the \emph{multiplicity $k_\s(y)$ in $\s$ of an element $y\in Y$} is the cardinality of the set $\s^{-1}(y)$.

\subsection{Irreducible correspondences}
The inclusion relation on $\cR(X,Y)$ specify a natural partial order: $R_1\le R_2$ iff $R_1\ss R_2$.

\begin{dfn}
The correspondences that are minimal w.r.t\. this order are called \emph{irreducible}, and all the remaining ones \emph{reducible}. By $\cR^0(X,Y)$ we denote the set of all irreducible correspondences between $X$ and $Y$ .
\end{dfn}

Thus, a correspondence is reducible iff it contains a pair $(x,y)$ which can be removed, but the relation remains to be a correspondence. Notice that the graph of a surjective mapping is an example of an irreducible correspondence.

\begin{prop}\label{prop:multiplicity}
A correspondence $R$ is irreducible iff for each $(x,y)\in R$ we have 
$$
\min\bigl(k_R(x),k_R(y)\bigr)=1.
$$
\end{prop}

\begin{proof}
To start with, suppose that $R$ is an irreducible correspondence, but for some $(x,y)\in R$ the both $k_R(x)$ and $k_R(y)$ are more than $1$. The latter means that there exist $x'\in X$, $x'\ne x$, and $y'\in Y$, $y'\ne y$, such that $(x',y)\in R$ and $(x,y')\in R$. Then $R\sm\bigl\{(x,y)\bigr\}\subset R$ is a correspondence, that contradicts to minimality of $R$.

For the converse, suppose that for each $(x,y)\in R$ it holds $\min\bigl(k_R(x),k_R(y)\bigr)=1$, but $R$ is not irreducible. Then there exists $(x,y)\in R$ such that $R\sm\bigl\{(x,y)\bigr\}$ is a correspondence as well. But then there exist $x'\in X$, $x'\ne x$, and $y'\in Y$, $y'\ne y$, such that $(x',y)\in R$ and $(x,y')\in R$, thus, $k_R(x)>1$ and $k_R(y)>1$, a contradiction.
\end{proof}

Below in this section $R$ denotes an irreducible correspondence between $X$ and $Y$, i.e., $R\in\cR^0(X,Y)$.

\begin{cor}\label{cor:nonreducible}
Choose arbitrary $x$ and $x'\ne x$ from $X$ and suppose that $k_R(x)>1$. Then $R(x)\cap R(x')=\0$.
\end{cor}

\begin{proof}
Suppose, to the contrary, that there exists $y\in Y$ such that $(x,y)\in R$ and $(x',y)\in R$. But then, for the pair $(x,y)\in R$, it simultaneously holds $k_R(y)>1$ and $k_R(x)>1$, that contradicts to Proposition~\ref{prop:multiplicity}.
\end{proof}

\begin{cor}\label{cor:nonred_property}
Let $\#X\ge2$ and $\#Y\ge 2$, then there is no $x\in X$ such that $\{x\}\x Y\ss R$.
\end{cor}

\begin{proof}
Assume the contrary. Consider $x'\in X$, $x\ne x'$, and any $y\in R(x')$, then the both multiplicities $k_R(x)$ and $k_R(y)$ are more than  $1$, a contradiction with Proposition~\ref{prop:multiplicity}.
\end{proof}

\begin{constr}
Let $R\in\cR^0(X,Y)$. For each $x\in X$ we put $X_1=\{x\in X:k_R(x)=1\}$ and $X_2=X\sm X_1$. Similarly, for each $y\in Y$ we put $Y_1=\{y\in Y:k_R(y)=1\}$ and $Y_2=Y\sm Y_1$.

By Proposition~\ref{prop:multiplicity}, we have $Y'_1:=R(X_2)\ss Y_1$ and $X'_1:=R^{-1}(Y_2)\ss X_1$. Further, put $X''_1=X_1\sm X'_1$ and $Y''_1=Y_1\sm Y'_1$. Thus, we constructed partitions
$$
X=X'_1\sqcup X''_1\sqcup X_2\quad \text{and}\quad  Y=Y'_1\sqcup Y''_1\sqcup Y_2.
$$
According to our construction, each $(x,y)\in R$ belongs exactly to one of the following sets:
$$
X'_1\x Y_2,\ X''_1\x Y''_1,\ X_2\x Y'_1.
$$
Moreover, the restriction of $R$ onto $X''_1\x Y''_1$ is one-to-one.

Further, for each $x\in X$ put $Y_x=R(x)$, and for each $y\in Y$ put $X_y=R^{-1}(y)$. By Corollary~\ref{cor:nonreducible}, for any distinct $x_1,\,x_2\in X_2$ the sets $Y_{x_1}$ and $Y_{x_2}$ do not intersect each other; similarly, for any distinct $y_1,\,y_2\in Y_2$ the sets $X_{y_1}$ and $X_{y_2}$ do not intersect each other as well. Put $\hX'_1=\{X_y\}_{y\in Y_2}$ and $\hY'_1=\{Y_x\}_{x\in X_2}$, then $\hX'_1$ and $\hY'_1$ form partitions of $X'_1$ and $Y'_1$, respectively, and the relation $R$ induces a bijection between $\hX'_1$ and $Y_2$, together with a bijection between $X_2$ and $\hY'_1$.
\end{constr}

Thus, we have proved the following result, illustrated by Figure~\ref{fig:IrreducibleRelation}.

\begin{thm}\label{thm:IrreducibleRelation}
Under the above notations, each irreducible correspondence $R\in\cR^0(X,Y)$ generates partitions $X=X'_1\sqcup X''_1\sqcup X_2$ and $Y=Y'_1\sqcup Y''_1\sqcup Y_2$, together with partitions $\hX'_1$ and $\hY'_1$ of the sets $X'_1$ and $Y'_1$, respectively. Also, the $R$ induces a bijection $\hR$ between the sets $\hX=\hX'_1\sqcup X''_1\sqcup X_2$ and $\hY=Y_2\sqcup Y''_1\sqcup\hY'_1$ such that $(x,y)\in R$ iff either $x\in\hx\in\hX'_1$, $y=\hR(\hx)\in Y_2$, or $x\in X''_1$, $y=\hR(y)\in Y''_1$, or $x\in X_2,\,y\in\hy=\hR(x)\in\hY'_1$.
\end{thm}

\ig{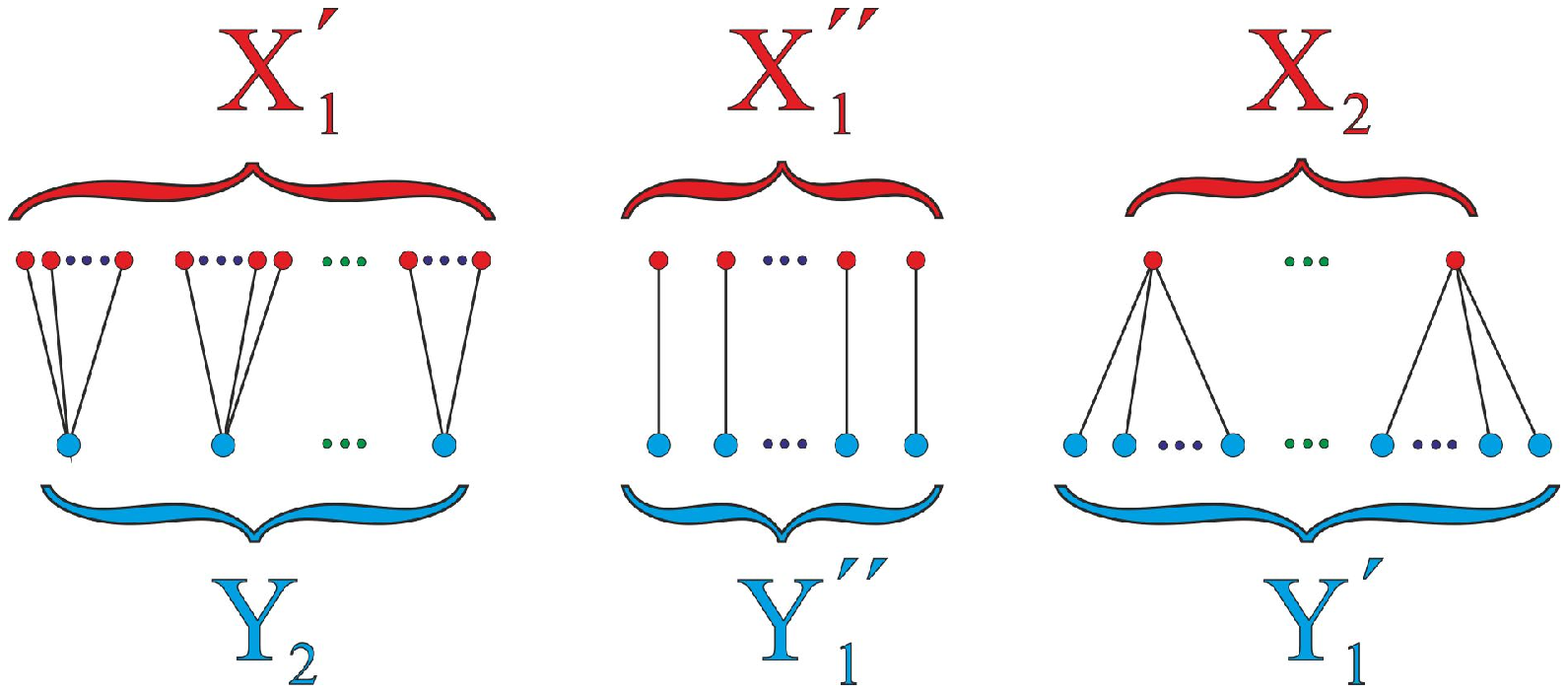}{0.25}{fig:IrreducibleRelation}{The structure of an irreducible correspondence.}

\begin{thm}\label{thm:sets_correspondences}
For each $R\in\cR(X,Y)$ there exists $R_0\in\cR^0(X,Y)$ such that $R_0\ss R$.
\end{thm}

\begin{rk}
One might try to use the technique based on Zorn's Lemma, but to do that one needs to guarantee that every chain $R_1\sp R_2\sp\cdots$ have a lower bound; in our specific case, the lower bound is a correspondence contained in all $R_i$. However, this is not true. As an example, consider $X=Y=\N$ and put $R_k=\bigl\{(i,j):\max(i,j)\ge k\bigr\}$. Clearly that every $R_k$ belongs to $\cR(X,Y)$, and that these $R_k$ form a descending chain. However, $\cap R_k=\0$ because for any $i$ and $j$ there exists $k$ such that $i<k$ and $j<k$, thus, $(i,j)\not\in R_k$.
\end{rk}

\begin{proof}[Proof of Theorem~$\ref{thm:sets_correspondences}$]
For each $x\in X$ we take any single element $(x,y)\in R$ and define a mapping $f\:X\to Y$ by the formula $y=f(x)$. Identifying a mapping with its graph, notice that $f\ss R$. Put $Y_1=f(X)$ and $Y_2=Y\sm Y_1$.

Now, for each $y\in Y_2$ we choose any single element $(x,y)\in R$ and define a mapping $g\:Y_2\to X$ by the formula $x=g(y)$. Again, identifying a mapping with its graph, notice that $g^{-1}\ss R$. Put $X_2=g(Y_2)$ and $X_1=X\sm X_2$.

Let $Y_3=f(X_2)$. Clearly that $Y_3\ss Y_1$, see Figure~\ref{fig:SubRelationConstruction}.

\ig{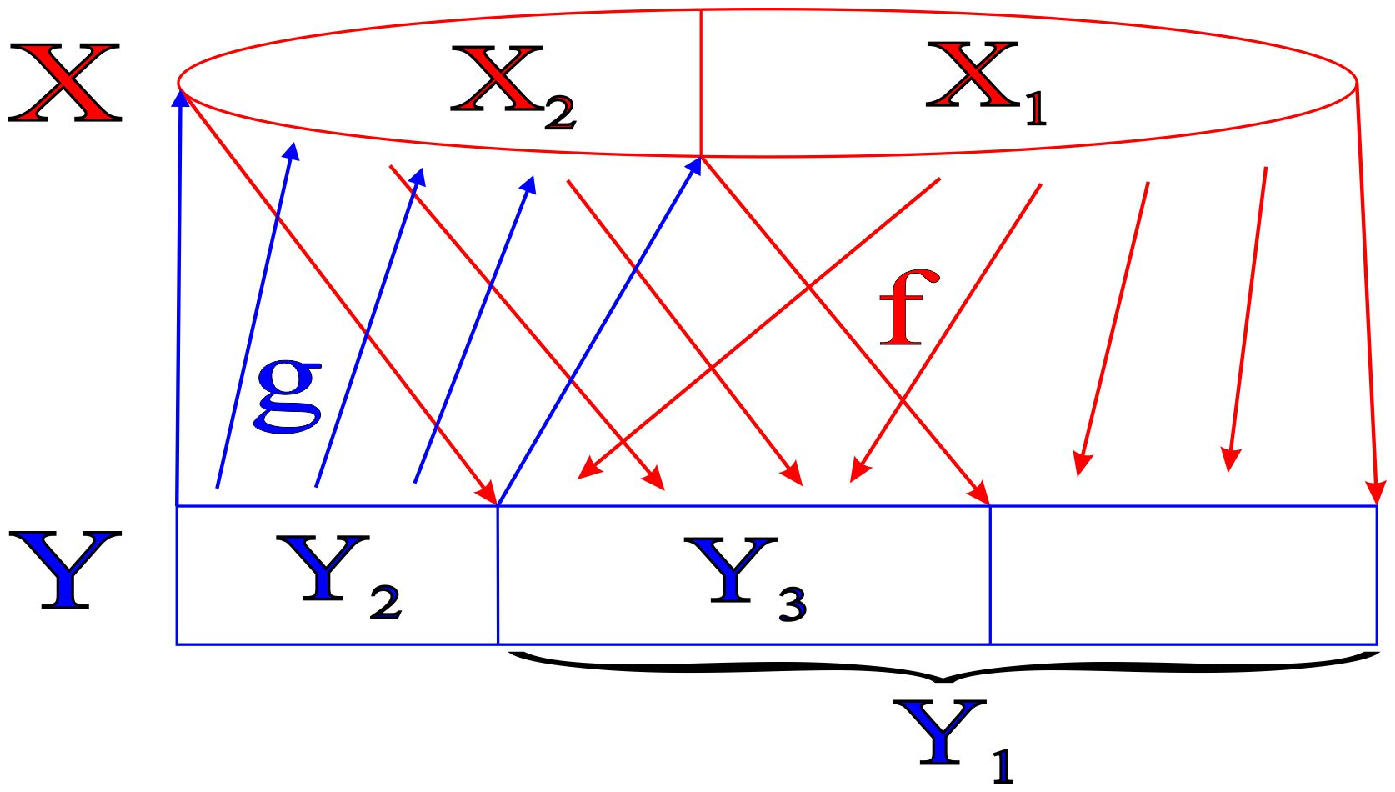}{0.25}{fig:SubRelationConstruction}{The construction of irreducible correspondence.}

Using $f$ and $g$, we define one more relation, namely, $h=f\cup g^{-1}$. Evidently, $h\in\cR(X,Y)$.

At last, we define a relation $R_0$ as follows. For each $y\in Y_3$ we remove from $h$ the following pairs $(x,y)$:
\begin{enumerate}
\item if $h^{-1}(y)\cap X_1\ne\0$, then we remove $\bigl(h^{-1}(y)\cap X_2\bigr)\x\{y\}$;
\item if $h^{-1}(y)\cap X_1=\0$, i.e., $h^{-1}(y)\ss X_2$, then we remove all elements belonging to $h^{-1}(y)\x\{y\}$ except any one.
\end{enumerate}

\begin{lem}
We have $R_0\in\cR(X,Y)$.
\end{lem}

\begin{proof}
For each $y\in Y\sm Y_3$ we have removed nothing, thus, for such $y$ there always exist $x\in X$ such that $(x,y)\in R_0$.

Now, let $y\in Y_3$. If $h^{-1}(y)\cap X_1\ne\0$, then we have only removed $(x,y)$ with $x\in X_2$, therefore, it remained an $x\in X_1$ for which the pair $(x,y)\in h$ was not removed and, thus, this pair belongs to $R_0$. If $h^{-1}(y)\cap X_1=\0$, then we have removed all the elements from $h^{-1}(y)\x\{y\}$, except any one, thus, for the $(x,y)\in h$ that was not removed we have $(x,y)\in R_0$. Thus, $\pi_Y(R_0)=Y$.

Now, let us pass to $\pi_X$. If $x\in X_1$, then the pairs $(x,y)\in h$ have not been not removed. If $x\in X_2$, then, since $X_2=\im g$, there exists $y\in Y_2$ such that $x=g(y)$, that implies $(x,y)\in h$, but those pairs has not bee removed from $h$, and hence, $(x,y)\in R_0$. Therefore, $\pi_X(R_0)=X$.
\end{proof}

\begin{lem}\label{lem:irred}
We have $R_0\in\cR^0(X,Y)$.
\end{lem}

\begin{proof}
It suffices to show that for every pair $(x,y)\in R_0$ either $x$, or $y$ is not contained in other pairs.

If $y\in Y_2$, then $x=g(y)$, and $y$ belongs exactly to the pair $\bigl(g(y),y\bigr)$. If $y\in Y_1\sm Y_3$, then $y=f(x)$, and $x$ belongs exactly to the pair $\bigl(x,f(x)\bigr)$.

At last, let $y\in Y_3$. If $h^{-1}(y)\cap X_1\ne\0$, then all pairs $(x',y)$ with $x'\in X_2$ have been removed, thus, $x\in X_1$; however, such $x$ belongs exactly to one pair, namely, to the $\bigl(x,f(x)\bigr)$. If $h^{-1}(y)\cap X_1=\0$, then we removed all the pairs $(x',y)$, $x'\in h^{-1}(y)\ss X_2$, except just one, so $y$ belongs to just one of such pairs.
\end{proof}

Lemma~\ref{lem:irred} completes the proof.
\end{proof}

The above reasonings imply the following result.

\begin{cor}\label{cor:dis_irreducible}
For any metric spaces $X$ and $Y$ we have
$$
d_{GH}(X,Y)=\frac12\inf\bigl\{\dis R:R\in\cR^0(X,Y)\bigr\}.
$$
\end{cor}

By $\cR^0_\opt(X,Y)$ we denote the set of all optimal irreducible correspondences between metric spaces $X$ and $Y$. Proposition~\ref{prop:optimal-correspondence-exists} and Theorem~\ref{thm:sets_correspondences}  immediately imply the following conclusion.

\begin{cor}\label{cor:irreducible-opt}
If $X,Y\in\cM$, then $\cR^0_\opt(X,Y)\ne\0$.
\end{cor}

Now we apply the technique of irreducible correspondences to geometry of the Gromov--Hausdorff space investigation.

\begin{examp}
Let $X=\{x_1,x_2\}$ and $Y=\{y_1,y_2\}$, then Corollary~\ref{cor:nonred_property} implies that $\cR^0(X,Y)$ consists only of bijections, so
$$
d_{GH}(X,Y)=\frac{\bigl||x_1x_2|-|y_1y_2|\bigr|}{2}.
$$
\emph{So, the family of isometry classes of $2$-point metric spaces, endowed with the Gromov--Hausdorff metric, is a metric space isometric to the open subray $x>0$ of the real line $\R$ with coordinate $x$}.
\end{examp}

\begin{examp}\label{examp:three-points-metric-spaces}
Consider two $3$-point metric spaces $X=\{x_1,x_2,x_3\}$ and $Y=\{y_1,y_2,y_3\}$. Put $r_{ij}=|x_ix_j|$ and $\r_{ij}=|y_iy_j|$. Let the lengths of the sides of the first triangle be equal to $a_1\le b_1\le c_1$, and the sides of the second one be equal to $a_2\le b_2\le c_2$ (the side of the length $a_k$ (respectively, $b_k$, $c_k$) is opposite to the first (second, third) vertex. We show that
$$
d_{GH}(X,Y)=\frac12\max\bigl\{|a_1-a_2|,\,|b_1-b_2|,\,|c_1-c_2|\bigr\}.
$$
To start with, let us describe all irreducible correspondences $R\in\cR^0(X,Y)$ which are not one-to-one. Without loss of generality, there exists a point $x\in X$ such that $\deg_R(x)>1$. By~\ref{cor:nonred_property}, $\deg_R(x)<3$, so $\deg_R(x)=2$. By~\ref{cor:nonreducible}, only the points from $X\sm\{x\}$ are in the relation with the point of the $1$-point set $Y\sm R(x)$. Put $x=x_i$, and  by $x_j$ and $x_k$ we denote the remaining points of $X\sm\{x\}$. Further, by $y_l$ and $y_m$ we denote the points from $R(x_i)$, and let $y_n$ be the remaining point from $Y$. Under those notatons the correspondence obtained above has the form
$$
R=\bigl\{(x_i,y_l),\,(x_i,y_m),\,(x_j,y_n),\,(x_k,y_n)\bigr\}.
$$
The distortion $\dis R$ of such correspondence equals
$$
\max\bigl\{r_{jk},\r_{lm},|r_{ij}-\r_{ln}|, |r_{ij}-\r_{mn}|,|r_{ik}-\r_{ln}|,|r_{ik}-\r_{mn}|\bigr\}.
$$
Notice that $\max\{r_{jk},\r_{lm}\}\ge|r_{jk}-\r_{lm}|$, so
$$
\dis R\ge\max\bigl\{|r_{jk}-\r_{lm}|,|r_{ij}-\r_{nl}|,|r_{ik}-\r_{nm}|\bigr\}=\dis R_1,
$$
where $R_1$ is the bijection $\bigl\{(x_i,y_n),(x_j,y_l),(x_k,y_m)\bigr\}$. Thus, to calculate $d_{GH}(X,Y)$ \emph{it suffices to consider only one-to-one correspondences}.

Notice that each bijection between $X$ and $Y$ generates a bijection between the sides of the triangles $X$ and $Y$, and the maximum of absolute values of differences between the lengths of the corresponding sides is just the distortion of this bijection. Thus, it is reasonable to consider the bijection $R$ as a one-to-one correspondence between the lengths of sides of the triangles $X$ and $Y$. Suppose that $R$ is a ``monotonic'' bijection, i.e., the one that preserves the order of the lengths of the corresponding sides. Then $\dis R=\max\bigl\{|a_1-a_2|,\,|b_1-b_2|,\,|c_1-c_2|\bigr\}$. Now we show that the distortion of any other bijection $R'$ is at least the same as of $R$, and this will complete the proof of the formula for $d_{GH}(X,Y)$ declared above.

Let $\dis R=|a_1-a_2|$. Without loss of generality, suppose that $a_1\le a_2$. Then either $(a_1,a_2)\in R'$ that implies $\dis R'\ge a_2-a_1=\dis R$, or $R'$ contains one of the pairs $(a_1,b_2)$, $(a_1,c_2)$, so again we have $\dis R'\ge a_2-a_1=\dis R$. In the case $\dis R=|c_1-c_2|$ the reasonings are quite similar.

Now, let $\dis R=|b_1-b_2|$. Again, without loss of generality, suppose that $b_1\le b_2$. If $(b_1,b_2)\in R'$, then $\dis R'\ge b_2-b_1=\dis R$. If either $(b_1,c_2)$, or $(a_1,b_2)$ belongs to $R'$, then again $\dis R'\ge\dis R$. Thus, it remains to consider the case $(b_1,a_2),\,(c_1,b_2)\in R'$, so $(a_1,c_2)\in R'$. But then $\dis R'\ge c_2-a_1\ge b_2-b_1=\dis R$, and the formula is completely proved.

\emph{Thus, the family of isometry classes of $3$-point metric spaces, endowed with the Gromov--Hausdorff metric, is a metric space which is isometric to a polyhedral cone}
$$
\bigl\{(a,b,c)\mid 0< a\le b\le c\le a+b\bigr\}
$$
\emph{in the space $\R^3_{\infty}$}. The corresponding isometry maps the triangle with the sides $a\le b\le c$ to the point $1/2(a,b,c)$.
\end{examp}

\section{Shortest networks in the Gromov--Hausdorff space and minimal fillings}
\markright{\thesection.~Shortest networks in the Gromov--Hausdorff space and minimal fillings}
In the present Section we apply~\ref{examp:three-points-metric-spaces} to investigate the Steiner problem in the Gromov--Hausdorff space $\cM$, as well as in its subspaces $\cM_n$, where  $\cM_n$ stands for the set of metric spaces containing at most $n$ points. Recall the main definitions.

Let $G=(V,E)$ be an arbitrary graph with the vertices set $V$ and the edges set $E$. We say that the graph $G$ is defined  \emph{on metric space $X$} if $V\ss X$. For every such a graph \emph{the length $|e|$} of any its \emph{edge $e=vw$} is defined as the distance $|vw|$, as well as \emph{the length $|G|$} of \emph{the graph $G$} as the sum of the lengths of all its edges.

If $M\ss X$ is an arbitrary finite subset and $G=(V,E)$ is a graph on $X$, then we say that the graph $G$ \emph{connects $M$} if $M\ss V$. The greatest lower bound of the lengths of connected graphs that connect $M\ss X$ is called \emph{the length of Steiner minimal tree on $M$}, or \emph{the length of shortest tree on $M$}, and is denoted by $\smt(M,X)$. Every connected graph $G$ connecting $M$ and such that $|G|=\smt(M,X)$ is a tree that is called \emph{Steiner minimal tree on $M$}, or \emph{shortest tree on $M$}. By $\SMT(M,X)$ we denote the set of all shortest trees on $M$. Notice that the set $\SMT(M,X)$ may be empty, and that $\SMT(M,X)$ and $\smt(M,X)$ depends on both the distances between points in $M$ and the geometry of the ambient space $X$: isometrical sets $M$ which belong to distinct metric spaces may be connected by shortest trees of non-equal lengths. For details on the shortest trees theory see for example~\cite{IvaTuz} or~\cite{HwRiW}.

\begin{prop}[\cite{IvaNikolaevaTuzSteiner}]\label{prop:Steiner-existens-GH}
For every $M\ss\cM_n$ we have $\SMT(M,\cM)\ne\0$. Moreover, if $M=\{m_1,\ldots,m_k\}$, $k\le n$, and $n_i$ denotes the number of points in the space $m_i$, $N=n_1+\cdots+n_k-(k-1)$, then there exists $G=(V,E)\in\SMT(M,\cM)$ such that $V\ss\cM_N$.
\end{prop}

\begin{rk}
The second part of~\ref{prop:Steiner-existens-GH} follows from the explicit construction described in the proof of Main Theorem in~\cite{IvaNikolaevaTuzSteiner}.
\end{rk}

Now, we fix a metric space $M$. Let us embed $M$ into various metric spaces $X$. ``The shortest length'' of Steiner minimal trees constructed on the images of $M$ we call \emph{the length of minimal filling for $M$} and denote it by $\mf(M)$. To avoid problems like Cantor's paradox, we give more accurate definition. The number $\mf(M)$ is the greatest lower bound of those $r$ for which there exist metric spaces $X$ and isometric embeddings $\mu\:M\to X$ such that $\smt\bigl(\mu(M),X\bigr)\le r$. For details on the theory of one-dimensional minimal fillings see~\cite{ITMinFil}.

The following result by Z.N.Ovsyannikov~\cite{Ovs} describes a relation between the minimal fillings and the shortest trees in $\R^n_\infty$. Let us start from necessarily definitions.  By $\ell_\infty$ we denote the space of all bounded sequences $x=\{x_i\}_{i=1}^\infty$ endowed with the norm $\|x\|=\max_i|x_i|$. Evidently, $\R^n_\infty$ can be isometrically embedded into $\ell_\infty$ as its coordinate subspace. Further, recall that a mapping $f\:X\to Y$ between metric spaces is called \emph{nonexpanding}, if $\big|f(x)f(x')\big|\le|xx'|$ for any $x,x'\in X$.

\begin{prop}[see~\cite{Day}]\label{ass:non_expand}
Let $M$ be a finite metric space, $K\subset M$ its nonempty subspace, and suppose that $f\:K\to\ell_\infty$ is an arbitrary nonexpanding mapping. Then the mapping $f$ can be extended up to a nonexpanding mapping from the entire space $M$ to $\ell_\infty$, i.e., there exists a nonexpanding mapping $g\:M\to\ell_\infty$ such that $g|_K=f$.
\end{prop}

Similar statement holds for a nonexpanding mapping to the finite-dimensional space $\R^n_\infty$.

\begin{cor}\label{cor:expand}
Let $M$ be a finite metric space, $K\subset M$ its nonempty subspace, and suppose that $f\:K\to\R^n_\infty$ is an arbitrary nonexpanding mapping. Then the mapping $f$ can be extended up to a nonexpanding mapping from the entire space $M$ to $\R^n_\infty$, i.e., there exists a nonexpanding mapping $g\:M\to\R^n_\infty$ such that $g|_K=f$.
\end{cor}

\begin{proof}
Indeed, let us identify the space $\R^n_\infty$ with a coordinate subspace of $\ell_\infty$. By~\ref{ass:non_expand}, we construct a nonexpanding mapping $g\:M\to\ell_\infty$ such that $g|_K=f$. Let us project $g(M)$ onto the coordinate subspace $\R^n_\infty\ss\ell_\infty$. Since the projection is identical on $\R^n_\infty$ and it does not expand distances, the mapping $M\to\R^n_\infty$ obtained as a result is that we sought for.
\end{proof}

\begin{prop}[Z.N.Ovsyannikov]\label{prop:Ovsyannikov}
Each shortest tree in $\R^n_\infty$ is a minimal filling for its boundary.
\end{prop}

\begin{proof}
Let $K\subset\R^n_\infty$ be an arbitrary nonempty finite subset. Consider an isometric Kuratowski embedding~\cite{ITMinFil} of this space into $\R^N_\infty$. By Corollary~3.3 from~\cite{ITMinFil}, its image $K'\subset\R^N$ can be connected by a shortest tree $\G'$ which is simultaneously a minimal filling of $K'$. By $M'\ss\R^N$ we denote the vertex set of $\G'$ with metric induced from $\R^N$. Consider the inverse isometric embedding $K'\to K\subset\R^n_\infty$. By Corollary~\ref{cor:expand}, it can be extended to an isometric embedding $f\:M'\to\R^n_\infty$. We put $M=f(M')$, and let $\G$ be a tree on $M$ for which $f$ is an isomorphism between $\G'$ and $\G$. The tree $\G$ connects $K$ and its length equals to the length of $\G'$, thus, $\G$ is a minimal filling for $K$ and, hence, $\G$ is also a Steiner minimal tree on $K$. Since all shortest trees on the same boundary have the same length, then all this trees are minimal fillings for $K$.
\end{proof}

In~\ref{examp:three-points-metric-spaces} we have constructed an isometrical mapping $\nu\:\cM_3\to\R^3_\infty$ whose image is the polyhedral cone
$$
\cC=\bigl\{(a,b,c)\mid 0\le a\le b\le c\le a+b\bigr\}\ss\R^3_\infty.
$$
The interior $\Int\cC$ of this cone corresponds to nondegenerate $3$-point metric spaces, i.e., to those spaces, where all the triangle inequalities hold strictly. It is easy to understand that for each point $P\in\Int\cC$ there exists a neighborhood $U_\e(P)\ss\R^3_\infty$ such that for any finite $M\ss U_\e(P)$ and any $G=(V,E)\in\SMT(M,\R^3_\infty)$ the inclusion $V\ss\Int\cC$  holds. Let us put $V'=\nu^{-1}(V)$, and construct $E'$ by taking exactly those $vw$ for which $\nu(v)\nu(w)\in E$. Thus, $G'$ is a tree on $\cM$. By~\ref{prop:Ovsyannikov}, the tree $G$ is a minimal filling, hence, due to that $\nu$ is isometrical, the tree $G'$ is a minimal filling as well, therefore, $G'\in\SMT\bigl(\nu^{-1}(M),\cM\bigr)$. Thus, we have proved the following result.

\begin{thm}
Let $T$ be an arbitrary $3$-point metric space all of whose triangle inequalities are strict. Then there exists $\e>0$ such that for any finite subset $M$ of the neighborhood $U_\e(T)\ss\cM_3$ each Steiner minimal tree $G\in\SMT(M,\cM)$ is a minimal filling of $M$.
\end{thm}

It turns out that for ``big'' boundary sets their Steiner minimal trees in $\cM$ may sometimes be not minimal fillings.

\begin{examp}
Consider $3$-point metric spaces $A^i=\{a^i_1,a^i_2,a^i_3\}$, $i=1,2,3$, for which the vectors $\nu(A^i)$ of nonzero distances are given by the following formulae:
\begin{equation}\label{eq:star}\tag{*}
B^1=\nu(A^1)=(8,\,22,\, 29.5),\ B^2=\nu(A^2)=(11.5,\,18,\,29),\ B^3=\nu(A^3)=(12,\,21.5,\,33). \\
\end{equation}
Let $S\in\R^3_\infty$ be the point with coordinates $(10,20,31)$. It does not satisfy to triangle inequalities, hence, it corresponds to neither $3$-point metric space, i.e., $\nu^{-1}(S)=\0$.

\begin{lem}[\cite{ITMinFil}]\label{lem:min-fil-tri}
For any metric space $X=\{P_1,P_2,P_3\}$, $\ell_i=|P_jP_k|$, $\{i,j,k\}=\{1,2,3\}$, its minimal filling is a star-like tree with the additional vertex $Z$ such that $|ZP_i|=(\ell_j+\ell_k-\ell_i)/2$, hence, the length of this minimal filling is equal to $(\ell_1+\ell_2+\ell_3)/2$, i.e., to the half-perimeter of the triangle $X$.
\end{lem}

Return to our example. Direct calculations give us $|B_iB_j|=4$, $|B_iS|=2$ in $\R^3_\infty$, thus, the total distance in $\R^3_\infty$ from $S$ to the points $B^i$ equals to half-perimeter of the triangle $B^1B^2B^3$, therefore, by~\ref{lem:min-fil-tri}, the tree $\bigl(\{B^1,B^2,B^3,S\},\{B^1S,B^2S,B^3S\}\bigr)$ is a minimal filling for the space $\{B^1,B^2,B^3\}\ss\R^3_\infty$, hence, it is also a shortest tree in $\R^3_\infty$ that connects $\{B^1,B^2,B^3\}$.

Let us show that a Steiner minimal tree for the boundary $M=\{A^1,A^2,A^3\}\in\cM$ (it exists due to~\ref{prop:Steiner-existens-GH}) is not a minimal filling for this boundary.

Indeed, if it is a minimal filling, then, by~\ref{lem:min-fil-tri}, there exists a point $Z\in\cM$ such that $d_{GH}(Z,A^i)=1$ for every $i=1,2,3$. Suppose that such point $Z$ does exist. By~\ref{prop:Steiner-existens-GH}, without loss of generality, one can suppose that $Z\in\cM_7$. Put $Z=\{z_1,\ldots,z_7\}$. Choose an arbitrary $R_i\in\cR^0_\opt(A_i,Z)$, then $\dis R_i=2$. Since for each $i$ the distances between distinct $a^i_j$ and $a^i_k$ are greater than $2$, then $R_i(a^i_j)\cap R_i(a^i_k)=\0$, hence, we have three partitions $Z=\sqcup_{k=1}^3R_i(a^i_k)$.

Further, the condition $\dis R_i=2$ implies that $\diam R_i(a^i_k)\le 2$ for all $i$ and $k$; moreover, since the least possible distance between distinct points of the space $B^i$ equals $8$, then the least possible distance between points from distinct $R_i(a^i_j)$ and $R_i(a^i_k)$ is not less than $6$. This implies that all partitions $\big\{R_i(a^i_k)\big\}_{k=1}^3$ are coincide.

Let $|a^1_1a^1_2|=29.5$, $|a^1_1a^1_3|=22$, $|a^1_2a^1_3|=8$, then for any $x\in R_1(a^1_1)$, $y\in R_1(a^1_2)$, and $z\in R_1(a^1_3)$ we have $|xy|\in[27.5,31.5]$, $|xz|\in[20,24]$, and $|yz|\in[6,10]$. Similarly construct the corresponding triplets of segment for $A^2$ and $A^3$. As a result, we get three triplets
\begin{equation}\tag{**}\label{eq:table}
\begin{array}{ccc}
[27.5,\,31.5],&[20,\,24],&[6,\,10], \cr
[27,\,31],&[16,\,20],&[9.5,\,13.5], \cr
[31,\,35],&[19.5,\,23.5],&[10,\,14].
\end{array}
\end{equation}
Notice that each of the three distances $|xy|$, $|yz|$, $|xz|$ has to belong to at least two other segments: one corresponding to $A^2$ and other one corresponding to $A^3$ --- the second and the third row in the table~(\ref{eq:table}). Thus, for each of these distances the corresponding three segments, one per each row, must have common intersection. It is easy to see that in the table~(\ref{eq:table}) the only segments which intersect each other are the ones in the same column. Thus, $|xy|\in[27.5,\,31.5]\cap[27,\,31]\cap[31,\,35]=\{31\}$, hence $|xy|=31$. Similarly we show that $|xz|=20$ and $|yz|=10$. Thus, there is a triplet of points in $Z$ that does not satisfy the triangle inequality, a contradiction.

In~\cite{ITMinFil} we have introduced the concept of \emph{Steiner subratio $\ssr(X)$} of a metric space $X$, namely,
$$
\ssr(X)=\inf\biggl\{\frac{\mf(M)}{\smt(M,X)}:M\ss X,\,2\le\#M<\infty\biggr\}.
$$

Taking into account the discussion above, we get the following conclusion.

\begin{cor}
For the boundary $\{A^1,\,A^2,\,A^3\}\ss\cM$, where $A^i$ are the $3$-point metric spaces, whose distances are given by the formulae~$(\ref{eq:star})$, its Steiner minimal tree in $\cM$ does exist and differs from minimal filling. Hence, $\ssr(\cM)<1$.
\end{cor}

\begin{rk}
We carried out a numerical experiment and got more precise estimate, namely, $\ssr(\cM)<0.857$. It would be interesting to get the exact value of the Steiner subratio for the Gromov--Hausdorff space.
\end{rk}
\end{examp}


\begin{thebibliography}{9}
\bibitem{IvaTuz} Ivanov A.O., Tuzhilin A.A. \emph{Minimal Networks. Steiner Problem and Its Generalizations}. CRC Press, 1994.
\bibitem{HwRiW} Hwang F.K., Richards D.S., Winter P. \emph{The Steiner Tree Problem}. Annals of Discrete Mathematics 53, North-Holland: Elsevier, 1992. ISBN 0-444-89098-X.
\bibitem{ITMinFil} Ivanov A.O., Tuzhilin A.A., ``One-dimensional Gromov Minimal Filling,'' ArXiv e-prints, {\tt arXiv:1101.0106} (2011).
\bibitem{BB} Bednov B.B., Borodin P.A., ``Banach Spaces that Realize Minimal Fillings,'' Sbornik: Math., {\bf 205} (4), pp.~459--475 ( 2014).
\bibitem{BurBurIva} Burago D., Burago Yu., Ivanov S. \emph{A Course in Metric Geometry}. Graduate Studies in Mathematics, vol.33. A.M.S., Providence, RI, 2001.
\bibitem{IvaIliadisTuz} Ivanov A.O., Iliadis S., Tuzhilin A.A., ``Realizations of Gromov-Hausdorff Distance,'' ArXiv e-prints, {\tt arXiv:1603.08850} (2016).
\bibitem{IvaNikolaevaTuzSteiner} Ivanov A.O., Nikolaeva N.K., Tuzhilin A.A., ``Steiner Problem in Gromov--Hausdorff Space: the case of finite metric spaces,'' ArXiv e-prints, {\tt arXiv:1604.02170}, (2016).
\bibitem{Ovs} Ovsyannikov Z.N. Course work, Faculty of Mechanics and Mathematics, Lomonosov Moscow State University,  Moscow, 2010.
\bibitem{Day} Day M.M., \emph{Normed linear spaces}, Springer, 1962.
\end{thebibliography}
\end{document}